\documentclass[10pt,reqno]{amsart}
\usepackage{amssymb,amsfonts}
\usepackage[all,arc]{xy}
\usepackage{enumerate}
\usepackage{mathrsfs}
\usepackage{geometry}
\usepackage{amsmath}
\usepackage{latexsym, bm}
\usepackage{indentfirst}
\usepackage{CJK}
\usepackage{amsthm}

\newtheorem{thm}{Theorem}[section]

\newtheorem{lem}{Lemma}[section]
\newtheorem{prop}{Proposition}[section]

\theoremstyle{definition}

\newtheorem{remk}{Remark}[section]
\DeclareMathOperator{\tr}{tr}
\DeclareMathOperator{\re}{Re}
\DeclareMathOperator{\Ric}{Ric}
\DeclareMathOperator{\Rm}{Rm}
\numberwithin{equation}{section}
\begin{document}
\title{Regularity of A Complex Monge-Amp\`{e}re Equation on Hermitian Manifolds}
\author{Xiaolan Nie \vspace{-4ex}}
\maketitle
\begin{abstract} We obtain higher order estimates for a parabolic flow on a compact Hermitian manifold. As an application, we prove that a bounded $\hat{\omega}$-plurisubharmonic solution of an elliptic complex Monge-Amp\`{e}re equation is smooth under an assumption on the background Hermitian metric $\hat{\omega}$. This generalizes a result of Sz\'{e}kelyhidi and Tosatti on K\"{a}hler manifolds.
\end{abstract}

\section{Introduction}

In [18], Sz\'{e}kelyhidi and Tosatti studied regularity of weak solutions of the equation
\begin{align}
(\omega+\sqrt{-1}\partial\bar{\partial}\phi)^n=e^{-F(\phi, z)}\omega^n
\end{align}
on an $n$-dimensional compact K\"{a}hler manifold $(M,\omega)$, where $F: \mathbb{R}\times M\rightarrow \mathbb{R}$ is a smooth function and $\omega+\sqrt{-1}\partial\bar{\partial}\phi\geq 0 $ in the sense of currents.

 According to the local theory of Bedford-Taylor [1], for a locally bounded plurisubharmonic function $u$, the wedge product $(dd^c u)^k,\ 1\leq k\leq n$ is well defined, where $d^c=\frac{\sqrt{-1}}{2}(\bar{\partial}-\partial)$ and $dd^c=\sqrt{-1}\partial\bar{\partial}$.
Indeed, for such $u$ and $T$ a positive closed current, the current $uT$ is well defined and
$$dd^cu\wedge T:= dd^c(uT)$$
 is also a positive closed current. Then the wedge product $(dd^c u)^k,\ 1\leq k \leq n$ can be defined inductively as closed positive currents. Denote
$$PSH(M,\omega)=\{u: M\rightarrow [-\infty, +\infty)| u\ \text{is upper semicontinuous},\ \omega+\sqrt{-1}\partial\bar{\partial}u\geq 0\} $$
 the set of $\omega$-plurisubharmonic (short for $\omega$-psh) functions on $M$. If $\phi\in PSH(M, \omega)\cap L^\infty(M)$ solves equation $(1.1)$ in the above sense of Bedford-Taylor, we say $\phi$ is a weak solution of the equation. The H\"{o}lder continuity of weak solutions follows from Ko{\l}odziej [12]. Using a different approach, Sz\'{e}kelyhidi and Tosatti [18] proved that such weak solutions are actually smooth. Particularly, if $M$ is Fano, $\omega\in c_1(M)$ and $F(\phi, z)=\phi-h$, where $h$ satisfies $\sqrt{-1}\partial\bar{\partial}h=\Ric(\omega)-\omega$, their result implies that K\"{a}hler-Einstein currents with bounded potentials are smooth.
\let\thefootnote\relax\footnote{\textbf{Mathematics Subject Classification (2010)} \ 53C55, 35J96}

 In the proof of [18], the authors use the smoothing property of the corresponding parabolic flow
\begin{align}
\frac{\partial \varphi}{\partial t}=\log\frac{(\omega + \sqrt{-1}\partial \bar{\partial}\varphi)^n}{\omega^n}+F(\varphi, z).
\end{align}
They construct a function $\varphi\in C^0([0,T]\times M)\cap C^{\infty}((0,T]\times M)$ with $\varphi(0)= \phi$ which solves equation (1.2) on $(0,T]$, where $T$ depends only on $\sup|\phi|$, $F$ and $\omega$. Then they show that $\dot{\varphi}(t)=0$ for $0<t\leq T$ since the initial $\phi$ is a solution of (1.1). Therefore $\phi=\varphi(0)=\varphi(t)$ is smooth. Similar construction was previously used in Song-Tian [17] for the K\"{a}hler-Ricci flow.\\

 As equation (1.1) also makes sense on Hermitian manifolds, it is natural to consider the regularity of weak solutions of equation (1.1) in a more general setting. On a Hermitian manifold, there are no local potentials for $\omega$. However, $\omega$-psh functions are locally the sum of plurisubharmonic functions and smooth functions. Using this property and the wedge product of a smooth positive $(1, 1)$ form and a positive current is again positive, the current $(\omega+\sqrt{-1}\partial\bar{\partial}\phi)^n$ is still well defined and positive for bounded $\omega$-psh functions. For more details on pluripotential theory we refer to [1, 3, 8, 9].\\

 In this note, we show that the higher order estimates in [18] can be obtained on compact Hermitian manifolds. Particularly, the flow $(1.2)$ with smooth initial data $\varphi_0$ has a smooth solution for a time $T$ which depends only on $\sup|\varphi_0|, \sup|\dot{\varphi}_0|$. Then we obtain the following theorem.

\begin{thm}
Let $(M,\hat{g})$ be a $n$-dimensional compact Hermitian manifold with the fundamental 2-form $\hat{\omega}$ satisfying
\begin{align}
\forall \ u\in PSH(M,\hat{\omega})\cap L^\infty(M),\ \ \int_M(\hat{\omega}+\partial\bar{\partial}u)^n=\int_M\hat{\omega}^n
 \end{align}
and $F:\mathbb{R}\times M\rightarrow \mathbb{R}$ be a smooth function. Suppose that $\phi\in PSH(M,\hat{\omega})\cap L^\infty(M)$ solves \begin{align}
 (\hat{\omega}+\sqrt{-1}\partial\bar{\partial}\phi)^n=e^{-F(\phi, z)}\hat{\omega}^n
\end{align}
in the sense of currents. Then $\phi$ is smooth.
\end{thm}

Here the assumption $(1.3)$ is automatically true on K\"{a}hler manifolds. In [18], the proof of the above theorem on K\"{a}hler manifolds needs Ko{\l}odziej's stability result [11]. We use the assumption $(1.3)$ from [3], under which the usual comparison principle is true, to make sure the stability result holds on such Hermitian manifolds [13]. Particularly, if $\hat{\omega}$ satisfies Guan-Li's [6] condition $\partial\bar{\partial} \hat{\omega}^k=0$, $k=1,2$, the assumption is satisfied. When $M$ is a complex surface, such metrics always exist due to a result of Gauduchon [4].

In the proof of our theorem, the main difference between the Hermitian case and K\"{a}hlar case lies in the $C^2, C^3$ estimates and bound for $|\Ric|$. The computation on Hermitian manifolds is more complicated due to the existence of torsion terms. The proof of the second order estimate follows closely the argument of Gill [5] and Tosatti-Weinkove [20]. For the third order estimate we make use of the arguments in Phong-\u{S}e\u{s}um-Sturm [14] and Sherman-Weinkove [19]. Such estimate for the first derivative of the evolving Hermitian metrics was also established in [24], where the authors took a local reference K\"{a}hler metric to obtain a good bound. To bound $|\Ric|$, we need to deal with the new terms involving $|\nabla\Ric|$ very carefully.

The techniques used in this paper can be applied to construct a weak solution of the Chern-Ricci flow [22, 23, 24] with singular initial Gauduchon metric on complex surfaces. In [22], Tosatti and Weinkove conjectured that if the Chern-Ricci flow starting from a Gauduchon metric is non-collapsing in finite time, then it blows down finitely many exceptional curves and continues in a unique way on a new complex surface. They proved in [23] the smooth convergence of the metrics away from the exceptional curves and the global Gromov-Hausdorff convergence (under a suitable condition) as $t$ approaches the singular time. It is expected that the flow can continue on the new surface from the push-down of the limiting current. We will investigate this in further work.

The paper is organized as follows. In section 2, we give some background material on Hermitian manifolds. Then we use maximum principle to obtain the estimates for existence of the parabolic flow for a short time depending only on $\sup |\varphi_0|$, $\sup |\dot{\varphi}_0|$. In section 4, we use the smoothing property of the parabolic flow to prove Theorem 1.1.\\

\noindent\textbf{Acknowledgements.} The author would like to thank her advisor Jiaping Wang for constant support, encouragement and many helpful discussions. The author also thanks Valentino Tosatti and Ben Weinkove for suggesting her the problem and helpful comments on the first version of this paper. In addition, the author is grateful to S{\l}awomir Ko{\l}odziej for the clarification on the stability result and to Haojie Chen for many useful conversations.

\section{Preliminaries}

For reader's convenience, in this section we introduce some basic material on Hermitian manifolds. The formulas given here can be found in [19].

Let $(M,g)$ be a $n$-dimensional compact Hermitian manifold with the fundamental 2-form $\omega=\sqrt{-1}g_{i\bar{j}}dz^i\wedge dz^{\bar{j}}$ in local coordinates. Denote $\nabla$ the Chern connection of $g$ with Christoffel symbols $\Gamma_{ij}^k$ and torsion $T$ given by:
$$\Gamma_{ij}^k=g^{k\bar{l}}\partial_i g_{j\bar{l}}, \ \ \ \ T_{ij}^k=\Gamma_{ij}^k-\Gamma_{ji}^k.$$ The covariant derivatives of $X=X^j\frac{\partial}{\partial z^j}$ and $a=a_jdz^j$ are defined in components as
$$\nabla_i X^j=\partial_i X^j+\Gamma^j_{ik}X^k, \ \ \nabla_i a_j=\partial_ia_j-\Gamma^k_{ij}a_k.$$
Then $\nabla$ can be extended naturally to any tensors.
Define the Chern curvature tensor of $g$ in components to be $$R_{i\bar{j}k}^{\ \ \ l}=-\partial_{\bar{j}}\Gamma^l_{ik}.$$ We lower and raise indices using metric g. Then $$R_{i\bar{j}k\bar{l}}=-\partial_i\partial_{\bar{j}}g_{k\bar{l}}+g^{p\bar{q}}\partial_ig_{k\bar{q}}\partial_{\bar{j}}g_{p\bar{l}}.$$ and the Chern-Ricci tensor is given by $$ R_{i\bar{j}}=g^{k\bar{l}}R_{i\bar{j}k\bar{l}}=-\partial_i\partial_{\bar{j}}\log \det g$$
We have the following commutation formulas:
\begin{align}\begin{split}
[\nabla_i, \nabla_{\bar{j}}]X^l=R_{i\bar{j}k}^{\ \ \ l}X^k, \ \ \ [\nabla_i, \nabla_{\bar{j}}]a_k=-R_{i\bar{j}k}^{\ \ \ l}a_l\\
[\nabla_i, \nabla_{\bar{j}}]\overline{X^l}=-R_{i\bar{j}\ \bar{k}}^{\ \ \bar{l}}\overline{X^k}, \ \ \ [\nabla_i, \nabla_{\bar{j}}]\overline{a_k}=R_{i\bar{j}\ \bar{k}}^{\ \ \bar{l}}\overline{a_l}\end{split}
\end{align}
The Bianchi identities will not hold necessarily for general Hermitian manifolds. There are extra torsion terms in the following identities.
\begingroup
\addtolength{\jot}{.2em}
\begin{align}
\begin{split}
R_{i\bar{j}k\bar{l}}-R_{k\bar{j}i\bar{l}}&=-\nabla_{\bar{j}}T_{ik\bar{l}}\\
R_{i\bar{j}k\bar{l}}-R_{i\bar{l}k\bar{j}}&=-\nabla_iT_{\bar{j}\bar{l}k}\\
R_{i\bar{j}k\bar{l}}-R_{k\bar{l}i\bar{j}}&=-\nabla_{\bar{j}}T_{ik\bar{l}}-\nabla_kT_{\bar{j}\bar{l}i}\\
\nabla_p R_{i\bar{j}k\bar{l}}-\nabla_i R_{p\bar{j}k\bar{l}}&=-T^{\ \ r}_{pi}R_{r\bar{j}k\bar{l}}\\
\nabla_{\bar{q}} R_{i\bar{j}k\bar{l}}-\nabla_{\bar{j}} R_{i\bar{q}k\bar{l}}&=-T^{\ \ \bar{s}}_{\bar{q}\bar{j}}R_{i\bar{s}k\bar{l}}.
\end{split}
\end{align}
\endgroup

\section{estimates for the parabolic flow}

Consider the following parabolic equation on a compact Hermitian manifold $(M ,\hat{\omega})$,
\begin{align}
\frac{\partial \varphi}{\partial t}=\log\frac{(\hat{\omega} + \sqrt{-1}\partial \bar{\partial}\varphi)^n}{\hat{\omega}^n}+F(\varphi, z)
\end{align}
where $F:\mathbb{R}\times M \rightarrow \mathbb{R}$ is a smooth function and $\varphi|_{t=0}=\varphi_0$ is smooth. By the theory of parabolic equations, there exists a unique smooth solution $\varphi(t)$ with $\hat{\omega} + \sqrt{-1}\partial \bar{\partial}\varphi>0$ for a short time. Denote $\dot{\varphi}$ for $\frac{\partial \varphi}{\partial t}$. We have the following proposition which generalizes the estimates in [18] to compact Hermitian manifolds.
\begin{prop}
Given a compact Hermitian manifold $(M,\hat{\omega})$, there exists $T>0$ depending only on $\sup|\varphi_0|$ and $F$ such that the above equation has a smooth solution $\varphi(t,z)$on $[0,T]$. Moreover, there exist smooth functions $C_k(t)$ on $(0,T]$ depending only on $\sup|\varphi_0|,\ \sup|\dot{\varphi}_0|,\  \hat{\omega}$ and $F$ which blow up as $t\rightarrow 0$ such that
$$\| \varphi(t)\|_{C^k(M)} < C_k(t)$$
for $t\leq T$.
\end{prop}

We write $g$ for the metric associated to $\omega=\hat{\omega} + \sqrt{-1}\partial \bar{\partial}\varphi$, where $\hat{\omega}$ is the background Hermitian metric on a compact complex manifold $M$. Denote $|\cdot|$ the norm of tensors with respect to $g$, $\nabla$ the Chern connection of $g$ and $\Delta=g^{p\bar{q}}\nabla_p\nabla_{\bar{q}}$ the Laplacian of $\nabla$. We use $\hat{\nabla},\ \hat{R}_{i\bar{j}},\ |\cdot|_{\hat{g}},\ \Delta_{\hat{g}}$ , etc. to denote the quantities associated to $\hat{\omega}$. Throughout the section, $C, C',c,c_i,...$ will be some constants which depend only on $\sup|\varphi_0|$, $\sup|\dot{\varphi}_0|$ (and $\hat{\omega}, F$), and may vary from line to line. Also we may denote $H$ to be different \vspace{1mm} quantities.\\

First we have the following lemma from [18].
\begin{lem}
There exist $T, C>0$ depending on $\sup|\varphi_0|$ such that
\vspace{1.6mm}
\begin{align}
|\varphi(t)|<C, \ \ \ |\dot{\varphi}(t)|\leq \sup|\dot{\varphi}(0)|e^{Ct},
\end{align}
when the solution exists and $t\leq T$. In particular,
\begin{align}
|\log\frac{(\hat{\omega} + \sqrt{-1}\partial \bar{\partial}\varphi)^n}{\hat{\omega}^n}|<C'
\end{align}
 for some $C'$ depending on $\sup|\varphi_0|$ and $ \sup|\dot{\varphi}_0|$.
\end{lem}
The proof follows from [18, Lemma 2.1] as it does not need the K\"{a}hler condition.
Now we can fix a $T'\leq T$ such that there exists a smooth\vspace{1.2mm} solution to $(3.1)$ on $[0,T']$.

The $C^1$ estimate in [18] was obtained by modifying B{\l}ocki's estimate [2] (see also [7], [15]). In Hermitian case, we need the following special local coordinate system from Guan-Li [6], which is also crucial for our second order estimate.
\begin{lem}
Around a point $p\in M$, there exist local coordinates such that at $p$,
\begin{align}
\hat{g}_{i\bar{j}}=\delta_{ij}, \ \ \ \ \frac{\partial \hat{g}_{i\bar{i}}}{\partial z_j}=0.
\end{align}
\end{lem}

With the above lemma, we have the following gradient estimate.

\begin{lem}
There exists $\alpha>0$ depending on $\sup|\varphi_0|$ and $\sup|\dot{\varphi}_0|$ such that
\begin{align}
|\nabla\varphi(t)|^2_{\hat{g}}<e^{\alpha/t},
\end{align}
for $t\leq T'$.
\end{lem}

\begin{proof} Define $$H=t\log|\nabla\varphi(t)|^2_{\hat{g}}-\gamma(\varphi),$$ where $\gamma$ is a smooth function which will be determined later. If $H$ achieves maximum on $[0,T']\times M$ at $t=0$, then $H$ is bounded by a constant depending on $F$ and $\sup |\varphi_0|$ by Lemma 3.1.
 Now assume $H$ achieves its maximum at a point $(t_0, z_0)$, $t_0>0$. Choose a coordinate system around $z_0$ in Lemma 3.2 such that $\varphi_{i\bar{j}}$ is diagonal at $z_0$. Write $\rho=|\nabla\varphi(t)|^2_{\hat{g}}=\hat{g}^{i\bar{j}}\varphi_i\varphi_{\bar{j}}$ and $\dot{\rho}=\frac{\partial\rho}{\partial t}$. As $(\frac{\partial}{\partial t}-\Delta )H
= \frac{\partial}{\partial t}H-t\frac{\Delta  \rho}{\rho}+t\frac{|\nabla \rho|^2}{\rho^2}+\Delta \gamma$, we do\vspace{.6mm} the following calculations at $z_0$. First we have
\begin{align*} \frac{\partial }{\partial t} H &=\log \rho+\frac{t \dot{\rho}}{\rho}-\gamma ' \dot{\varphi}\\
&=\log \rho-\gamma ' \dot{\varphi}+ 2\frac{t}{\rho} (\sum_{i,k}\re(\frac {\varphi_{ki\bar{i}}\varphi_{\bar{k}}} {1+\varphi_{i\bar{i}}})+2F' \sum_i|\varphi_i|^2+2\sum_i\re(F_i \varphi_i)),
\end{align*}
where the second equality follows from
\begingroup
\addtolength{\jot}{.3em}
\begin{align*}
\dot{\rho}&=\sum_i\dot{\varphi}_i\varphi_{\bar{i}}+\varphi_i\dot{\varphi}_{\bar{i}}\\
\dot{\varphi_i}&=g^{k\bar{l}}(\partial_i \hat{g}_{k\bar{l}}+\varphi_{i\bar{l}k})-\hat{g}^{k\bar{l}}\partial_i \hat{g}_{k\bar{l}}+F'\varphi_i+F_i\\
&=\sum_k\frac{\varphi_{ik\bar{k}}}{1+\varphi_{k\bar{k}}}+F'\varphi_i+F_i.\end{align*}\endgroup
Here $F'$ is the derivative in the $\varphi$ direction. Also, using $\partial_i\partial_{\bar{i}}\hat{g}^{k\bar{l}}=-\partial_i\partial_{\bar{i}} \hat{g}_{l\bar{k}}+\sum_q
\partial_i \hat{g}_{q\bar{k}}\partial_{\bar{i}} \hat{g}_{l\bar{q}}+\sum_p
\partial_i \hat{g}_{l\bar{p}}\partial_{\bar{i}} \hat{g}_{p\bar{k}}$, \vspace{-1.5mm}
we get
\begin{align*}
\Delta \rho &=g^{i\bar{i}}\partial_i \partial_{\bar{i}}(\hat{g}^{k\bar{l}}\varphi_k\varphi_{\bar{l}})\\
&=\sum_{i,k}\frac{1}{1+\varphi_{i\bar{i}}}(-\sum_l\partial_i \partial_{\bar{i}} \hat{g}_{l\bar{k}}\varphi_k\varphi_{\bar{l}}+2\re(\varphi_{k\bar{i}i}\varphi_{\bar{k}})\\
&\ \ \ +|\varphi_{k\bar{i}}-\sum_l\partial_{\bar{i}}\hat{g}_{\bar{l}k}\varphi_l|^2
+|\varphi_{ki}-\sum_l\partial_{i}\hat{g}_{k\bar{l}}\varphi_l|^2).
\end{align*}
At $(t_0,z_0)$, $\nabla H=0$\vspace{-2mm} gives
\begin{align}H_i=\frac{t}{\rho}\rho_i-\gamma'\varphi_i=0.\end{align}
The\vspace{-1mm}n
$$\frac{|\nabla \rho|^2}{\rho^2}=\sum_i \frac{1}{1+\varphi_{i\bar{i}}}(\frac{\gamma'}{t})^2|\varphi_i|^2.$$
Also $\Delta \gamma (\varphi)=
\underset{i}{\sum}\frac{1}{1+\varphi_{i\bar{i}}}(\gamma''|\varphi_i|^2+\gamma'\varphi_{\bar{i}i}).$
Therefore\vspace{-1mm} we get
\begin{align*}
0 \leq\  &(\frac{\partial}{\partial t}-\Delta )H\\
=\ &\frac{\partial}{\partial t}H-t\frac{\Delta  \rho}{\rho}+t\frac{|\nabla \rho|^2_{g}}{\rho^2}+\Delta  \gamma\\
\leq\ & \log \rho- \gamma' \dot{\varphi} +ct-\sum_{i,k}\frac{t}{\rho}\frac{1}{1+\varphi_{i\bar{i}}}(|\varphi_{k\bar{i}}-\sum_l\partial_{\bar{i}}\hat{g}_{\bar{l}k}\varphi_l|^2
+|\varphi_{ki}-\sum_l\partial_{i}\hat{g}_{k\bar{l}}\varphi_l|^2)\\
&+\sum_i\frac{|\varphi_i|^2}{1+\varphi_{i\bar{i}}}(\frac{(\gamma ')^2}{t}+\gamma'' )+ \sum_i\frac{c_1 t-\gamma'}{1+\varphi_{i\bar{i}}}
+n\gamma'+\frac{c_2t}{\rho}
\end{align*}
Now we use the same trick in [2, 18] to control the term containing $\gamma'^2$. From (3.6) we get
\begin{align*}
\gamma'\rho\varphi_i=t\rho_i=t(\varphi_i\varphi_{i\bar{i}}+\sum_k\varphi_{ki}\varphi_{\bar{k}}-\sum_{k,l}\partial_i \hat{g}_{l\bar{k}}\varphi_k\varphi_{\bar{l}})
 \end{align*}
which gives  $$\sum_k(\varphi_{ki}-\sum_l\partial_{i}\hat{g}_{k\bar{l}}\varphi_l)\varphi_{\bar{k}}=t^{-1}\gamma' \rho \varphi_i-\varphi_i\varphi_{i\bar{i}}.$$
So
\begin{align*}
\frac{t}{\rho}\sum_{i,k}\frac{|\varphi_{ki}-\sum_l\partial_{i}\hat{g}_{k\bar{l}}\varphi_l|^2}{1+\varphi_{i\bar{i}}}&\geq \frac{t}{\rho^2}\sum_i\frac{|\sum_k(\varphi_{ki}-\sum_l\partial_{i}\hat{g}_{k\bar{l}}\varphi_l)\varphi_{\bar{k}}|^2}{1+\varphi_{i\bar{i}}}\\
&=\frac{t}{\rho^2}\sum_i\frac{ |t^{-1}\gamma' \rho \varphi_i-\varphi_i\varphi_{i\bar{i}}|^2}{1+\varphi_{i\bar{i}}}\\
&\geq \frac{(\gamma')^2}{t}\sum_i\frac{|\varphi_i|^2}{1+\varphi_{i\bar{i}}}-2\gamma'
\end{align*}
where we assume $\gamma'>0$. As $\dot{\varphi}$ is bounded from Lemma 3.1 for $t\leq T'$, the above estimates gives
$$0 \leq \log \rho+ct+\sum_i\frac{\gamma'' |\varphi_i|^2}{1+\varphi_{i\bar{i}}}+ \sum_i\frac{c_1 t-\gamma'}{1+\varphi_{i\bar{i}}}
+(n+2+c)\gamma'+\frac{c_2t}{\rho}$$
Take $\gamma(x)=Ax-\frac{1}{A}x^2$. Assume that $\log \rho\geq 1$ at $ (t_0, z_0)$ and choose $A$ to be sufficiently large, then we get $$\sum_i\frac{|\varphi_i|^2}{1+\varphi_{i\bar{i}}}+ \sum_i\frac{1}{1+\varphi_{i\bar{i}}}\leq c' \log \rho$$ for some constan $c'$. The above inequality together with $(3.3)$ imply that
 $$ 1+\varphi_{i\bar{i}}\leq c(c'\log\rho)^{n-1},$$ Then we have $$\rho=\sum_i|\varphi_i|^2\leq nc(c'\log\rho)^n,$$ which shows that $\rho$ is bounded at $(t_0, z_0)$. Therefore $H$  has a bound depending only on $\sup|\varphi_0|$, $\sup|\dot{\varphi}_0|$ and the estimate $(3.5)$ follows.
\end{proof}
Now we will give the second order estimate.  We use the idea of [6, 20] and follow the argument in [5] closely. For local computations in the proof of the following proposition, we always use a coordinate system in Lemma 3.2 at a point $p$, such that  $\hat{g}_{i\bar{j}}=\delta_{ij}, \ \frac{\partial \hat{g}_{i\bar{i}}}{\partial z_j}=0.$ and $\varphi_{i\bar{j}}$ is diagonal.
\begin{prop}
There exists $C>0$ depending on $\sup|\varphi_0|$ and $\sup|\dot{\varphi}_0|$ such that
\begin{align}
\tr _{\hat{g}}g=n+\Delta \varphi(t)<e^{Ce^{\alpha/t}}
\end{align}
for \vspace{-1mm} $t\leq T'$, where $\alpha$ is the same as in Lemma $3.3$.
\end{prop}

\begin{proof}
Let $$ H=e^{-\frac{\alpha}{t}}\log \tr _{\hat{g}}g+e^\Psi,$$
where $\Psi=A(\underset{[0,T']\times M}{\sup \varphi}-\varphi)$ and $A$ is a constant to be chosen later. First we have \begingroup
\addtolength{\jot}{.1em}
\begin{align}
(\frac{\partial}{\partial t}-\Delta )H&= \frac{\alpha}{t^2}e^{-\frac{\alpha}{t}}\log\tr_{\hat{g}}g+\frac{e^{-\frac{\alpha}{t}}}{{\tr_{\hat{g}}g}}\Delta_{\hat{g}} \dot{\varphi}-Ae^{\Psi}\dot{\varphi}\nonumber \\
&-e^{-\frac{\alpha}{t}}\Delta \log \tr_{\hat{g}}g-A^2|\nabla{\varphi}|^2e^{\Psi}-A(\tr_g{\hat{g}}- n)e^{\Psi}.
\end{align}\endgroup
It follows from $(3.1)$ that
\begin{align}\Delta_{\hat{g}} \dot{\varphi}=-\tr_{\hat{g}}\Ric(g)+\tr_{\hat{g}}\Ric(\hat{g})+\Delta_{\hat{g}} F(\varphi,z)\end{align}
where $$\Delta_{\hat{g}} F(\varphi,z)=F''|\nabla \varphi|^2_{\hat{g}}+F'\Delta_{\hat{g}} \varphi+2\re(g^{i\bar{j}}F'_i \varphi_{\bar{j}})+\Delta_{\hat{g}} F.$$ Here $F'$ is the derivative in the $\varphi$ direction, $\Delta_{\hat{g}}$ is the complex Laplacian of $F$ in the $z$ variable. Use that
\begin{align*}\tr_{\hat{g}}\Ric(g)&=\sum_{i,k} g^{i\bar{i}}(-\partial_k \partial_{\bar{k}} g_{i\bar{i}}+g^{j\bar{j}}\partial_kg_{i\bar{j}}\partial_{\bar{k}} g_{j\bar{i}})\\
&=\sum_{i,k} g^{i\bar{i}}(-\varphi_{i\bar{i}k\bar{k}}-\partial_k \partial_{\bar{k}} \hat{g}_{i\bar{i}}+g^{j\bar{j}}\partial_kg_{i\bar{j}}\partial_{\bar{k}} g_{j\bar{i}})
\end{align*}
to rewrite $(3.9)$ as
\begin{align}
\sum_{i,k} g^{i\bar{i}}\varphi_{i\bar{i}k\bar{k}}&=-\sum_{i,k}g^{i\bar{i}}\partial_k \partial_{\bar{k}} {\hat{g}}_{i\bar{i}}+\sum_{i,j,k}g^{i\bar{i}}g^{j\bar{j}}\partial_kg_{i\bar{j}}\partial_{\bar{k}} g_{j\bar{i}}
+\Delta_{\hat{g}} \dot{\varphi}-\tr_{\hat{g}}\Ric(\hat{g})-\Delta_{\hat{g}} F(\varphi,z)\nonumber\\
&\geq \sum_{i,j,k}g^{i\bar{i}}g^{j\bar{j}}\partial_kg_{i\bar{j}}\partial_{\bar{k}} g_{j\bar{i}}
+\Delta_{\hat{g}} \dot{\varphi}-C_1 |\nabla\varphi|^2_{\hat{g}}-C_2\tr_{\hat{g}}g\tr_{g}\hat{g}.
\end{align}
From the bound in $(3.3)$, we have $\tr_{\hat{g}}g,\  \tr_g\hat{g}\geq C^{-1}$ for some constant $C$ and then $\tr_{\hat{g}}g,\ \tr_g\hat{g}\leq C\tr_g\hat{g}\tr_{\hat{g}}g$. We use these in the above inequality. Also we will use them frequently in the following.
As the estimates in [20,\ (2.6)], we have \begin{align}\Delta \tr_{\hat{g}}g\geq \sum_{i,k} g^{i\bar{i}}\varphi_{i\bar{i}k\bar{k}}-2\re(\sum_{i,j,k}g^{i\bar{i}}\partial_{\bar{i}} \hat{g}_{j\bar{k}}\varphi_{k\bar{j}i})-C\tr_{\hat{g}}g\tr_{g}\hat{g}. \end{align}
To control $\sum_{i,j,k}g^{i\bar{i}}\partial_{\bar{i}} \hat{g}_{j\bar{k}}\varphi_{k\bar{j}i}$, we use a trick from [6].
$$\sum_{i,j,k}g^{i\bar{i}}\partial_{\bar{i}} \hat{g}_{j\bar{k}}\varphi_{k\bar{j}i}=\sum_i\sum_{j\neq k}(g^{i\bar{i}}\partial_{\bar{i}} \hat{g}_{j\bar{k}}\partial_k g_{i\bar{j}}-g^{i\bar{i}}\partial_{\bar{i}} \hat{g}_{j\bar{k}}\partial_{k} \hat{g}_{i\bar{j}}).$$
so
\begin{align}
|2\re (\sum_{i,j,k} g^{i\bar{i}}\partial_i \hat{g}_{j\bar{k}}\varphi_{k\bar{j}i}|&\leq \sum_i\sum_{j\neq k} (g^{i\bar{i}}g^{j\bar{j}}\partial_kg_{i\bar{j}}\partial_{\bar{k}} g_{j\bar{i}}+g^{i\bar{i}}g_{j\bar{j}}\partial_{\bar{i}}\hat{g}_{j\bar{k}}\partial_i \hat{g}_{k\bar{j}})+C\tr_{g}\hat{g}\nonumber \\
&\leq \sum_i\sum_{j\neq k} g^{i\bar{i}}g^{j\bar{j}}\partial_kg_{i\bar{j}}\partial_{\bar{k}} g_{j\bar{i}}+C\tr_{g}\hat{g} \tr_{\hat{g}}g.
\end{align}
Combing $(3.10),(3.11),(3.12)$ we can get
$$\Delta \tr_{\hat{g}}g\geq \sum_{i,j}g^{i\bar{i}}g^{j\bar{j}}\partial_jg_{i\bar{j}}\partial_{\bar{j}} g_{j\bar{i}}+\Delta_{\hat{g}}\dot{\varphi}-C_1 |\nabla\varphi|^2_{\hat{g}}-C\tr_{g}\hat{g}\tr_{\hat{g}}g$$
Now we will control $\frac{|\partial \tr_{\hat{g}}g|^2}{(\tr_{\hat{g}}g)^2}$. As
$$\partial_i \tr_{\hat{g}}g=\partial_i \sum_j \varphi_{j\bar{j}}=\sum_j \partial_j \varphi_{i\bar{j}}=\sum_j(\partial_j g_{i\bar{j}}-\partial_j\hat{g}_{i\bar{j}}),$$
then
$$\frac{|\partial \tr_{\hat{g}}g|^2}{(\tr_{\hat{g}}g)^2}\leq \frac{1}{(\tr_{\hat{g}}g)^2}\sum_{i,j,k} g^{i\bar{i}}\partial_jg_{i\bar{j}}\partial_{\bar{k}} g_{k\bar{i}}-\frac{2}{(\tr_{\hat{g}}g)^2}\re(\sum_{i,j,k} g^{i\bar{i}}\partial_j\hat{g}_{i\bar{j}}\partial_{\bar{k}} g_{k\bar{i}})+C\tr_{g}\hat{g}.$$
Assume that $H$ achieves maximum at $(t_0, z_0),\ t_0>0$, then $\nabla H(t_0, z_0)=0$ gives
$$\frac{e^{-\frac{\alpha}{t}}\partial_{\bar{i}} \tr_{\hat{g}}g}{\tr_{\hat{g}}g}-Ae^{\Psi}\varphi_{\bar{i}}=0$$
That is, $$\sum_k\partial_{\bar{i}}g_{k\bar{k}}=Ae^{\frac{\alpha}{t}}\tr_{\hat{g}}g\varphi_{\bar{i}}e^{\Psi}$$
Together with $\partial_{\bar{k}} g_{k\bar{i}}=\partial_{\bar{k}} \hat{g}_{k\bar{i}}+\partial_{\bar{i}} g_{k\bar{k}}$, we get
\begingroup
\addtolength{\jot}{.4em}
\begin{align}
|\frac{2}{(\tr_{\hat{g}}g)^2}\re(\sum_{i,j,k} g^{i\bar{i}}\partial_j{\hat{g}}_{i\bar{j}}\partial_{\bar{k}} g_{k\bar{i}})|
&\leq  |\frac{2Ae^{\frac{\alpha}{t}}e^{\Psi}}{\tr_{\hat{g}}g}\re\sum_{i,j} g^{i\bar{i}}\partial_j{\hat{g}}_{i\bar{j}}\varphi_{\bar{i}}|+C\tr_{g}\hat{g}\nonumber\\
&\leq e^{\frac{\alpha}{t}}e^{\Psi}(A^2|\nabla \varphi|^2+\frac{C\tr_g\hat{g}}{(\tr_{\hat{g}}g)^2})+C\tr_g\hat{g} \nonumber\\
&\leq e^{\frac{\alpha}{t}}e^{\Psi}(A^2|\nabla \varphi|^2+C'\tr_{g}\hat{g})+C\tr_{g}\hat{g}
\end{align}\endgroup
Using the Cauchy-Schwarz inequality as in Yau's second order estimate [25] (see equation (2.21) in [20]), we have
\begin{align}
\frac{1}{\tr_{\hat{g}}g}\sum_{i,j,k}g^{i\bar{i}}\partial_jg_{i\bar{j}}\partial_{\bar{k}} g_{k\bar{i}}\leq \sum_{i,j} g^{i\bar{i}}g^{j\bar{j}}\partial_jg_{i\bar{j}}\partial_{\bar{j}} g_{j\bar{i}} \hspace{4mm}
\end{align}
Combing $(3.13)$ and $(3.14)$, we \vspace{-.5mm} get
\begingroup
\addtolength{\jot}{.3em}
\begin{align*}
\frac{|\partial \tr_{\hat{g}}g|^2}{(\tr_{\hat{g}}g)^2}&\ \leq\frac{1}{\tr_{\hat{g}}g}\sum_{i,j} g^{i\bar{i}}g^{j\bar{j}}\partial_jg_{i\bar{j}}\partial_{\bar{j}} g_{j\bar{i}}+ e^{\frac{\alpha}{t}}e^{\Psi}(A^2|\nabla \varphi|^2+C'\tr_{g}\hat{g})+C\tr_{g}\hat{g}.
\end{align*}\endgroup
\begingroup
\addtolength{\jot}{.4em}
So \begin{align}
\begin{split}
e^{-\frac{\alpha}{t}}\Delta \log\tr_{\hat{g}}g=&e^{-\frac{\alpha}{t}}(\frac{\Delta \tr_{\hat{g}}g}{\tr_{\hat{g}}g}-\frac{|\partial \tr_{\hat{g}}g|^2}{(\tr_{\hat{g}}g)^2} )\\
\geq\ & \frac{e^{-\frac{\alpha}{t}}}{\tr_{\hat{g}}g}\Delta_{\hat{g}} \dot{\varphi}-C_1\frac{e^{-\frac{\alpha}{t}}}{\tr_{\hat{g}}g}|\nabla\varphi|^2_{\hat{g}}-Ce^{-\frac{\alpha}{t}}\tr_{g}\hat{g}\\
&-A^2|\nabla \varphi|^2e^{\Psi}-C'\tr_{g}\hat{g}e^{\Psi}
\end{split}
\end{align}
\endgroup
From $(3.3)$ we have $\tr_{\hat{g}}g\leq \ C(\tr_{g}\hat{g})^{n-1}$ for some constant $C$. Now put $(3.15)$ into $(3.8)$ and using that $\varphi, \ \dot{\varphi}$ is bounded and $(3.5)$, we have
\begingroup
\addtolength{\jot}{.5em}
\begin{align*}
(\frac{\partial}{\partial t}-\Delta )H\leq\  &C\log\tr_{\hat{g}}g-A\dot{\varphi}e^{\Psi}+C_1+Ce^{-\frac{\alpha}{t}}\tr_{g}\hat{g}\\
&+C\tr_{g}\hat{g}e^{\Psi}+Ane^{\Psi}-A\tr_{g}\hat{g}e^{\Psi}\\
\leq\  &C'\log\tr_{\hat{g}}g+AC'e^{\Psi}-(A-C)e^{\Psi}\tr_{g}\hat{g}+Ce^{-\frac{\alpha}{t}}\tr_{g}\hat{g}\\
\leq\ &-(A-C-C_1)e^{\Psi}\tr_{g}\hat{g}+ AC'e^{\Psi}
\end{align*}
\endgroup
Choosing $A$ large enough such that $A-C-C_1\geq 0$, then at $(t_0, z_0),$
$$0\leq -(A-C-C_1)\tr_{g}\hat{g}+AC'.$$
for $t\leq T'$ gives $\tr_{g}\hat{g}\leq C'$ at $(t_0, z_0)$, which implies that $H\leq C$ for some constant $C$ depending on $\sup|\varphi_0|$ and $\sup|\dot{\varphi}_0|$. Then we obtain the desired estimate $(3.7)$.
\end{proof}
Now we give the third order estimate. Our proof is based on the arguments in [14, 19]( see also [24]).  As in [25], consider $S=g^{i\bar{p}}g^{q\bar{j}}g^{k\bar{r}}\varphi_{i\bar{j}k}\varphi_{\bar{p}q\bar{r}}$
where $\varphi_{i\bar{j}k}=\hat{\nabla}_k\varphi_{i\bar{j}}$. We introduce the tensor $\Phi_{ij}^{\ \ k}=\Gamma_{ij}^k-\hat{\Gamma}_{ij}^k$ and then $$S=|\Phi|^2=g^{i\bar{p}}g^{j\bar{q}}g_{k\bar{r}}\Phi_{ij}^{\ \ k}\Phi_{\bar{p}\bar{q}}^{\ \ \bar{r}}.$$
From now on, we will write $k(t), k_1(t), k_2(t)$, ... for a function of the form $Ke^{\lambda Ce^{\alpha/t}}$ where $e^{Ce^{\alpha/t}}$ is the bound in Proposition $3.2$, and $K, \lambda$ are constants depending only on $\hat{\omega}, F$. In the proof of the following proposition, We will use the estimates $|\nabla\varphi(t)|^2_{\hat{g}}\leq k(t),\ \tr _{\hat{g}}g\leq k(t)$ repeatedly.

\begin{prop}
There exists a smooth function $C(t)>0$ on $(0,T']$ depending only on $\sup |\varphi_0|,\ \sup |\dot{\varphi}_0|$ and blowing up as $t\rightarrow 0$ such that $S<C(t)$ for $t\leq T'$.
\end{prop}

\begin{proof}
As the calculations in [14, 19], first we have
\begin{align*}
\Delta S=&|\overline{\nabla}\Phi|^2+|\nabla\Phi|^2-\Phi_{ij}^{\ \ k}\left(R^{\ p\ q}_{p\ k}\Phi^{ij}_{\ \ q}-R_{p\ q}^{\ p\ i}\Phi_{\ \ k}^{qj}-R_{p\ q}^{\ p\ j}\Phi^{iq}_{\ \ k}\right)\\
&+2\re \left(\Delta \Phi_{ij}^{\ \ k}\Phi_{\ \ k}^{ij}\right),\\
\frac{\partial}{\partial t}S=&\Phi_{ij}^{\ \ k}\left(\frac{\partial}{\partial t} g^{i\bar{q}}\Phi_{\bar{q}\ k}^{\ j}+\frac{\partial}{\partial t} g_{k\bar{q}}\Phi^{ij\bar{q}}+\frac{\partial}{\partial t} g^{j\bar{q}}\Phi^i_{\ \bar{q}k}\right)+2\re \left(\frac{\partial}{\partial t}\Phi_{ij}^{\ \ k} \Phi^{ij}_{\ k}\right)\\
=&\Phi_{ij}^{\ \ k}\left(g^{q\bar{r}}\frac{\partial}{\partial t}g_{k\bar{r}}\Phi^{ij}_{\ \ q}-g^{i\bar{r}}\frac{\partial}{\partial t}g_{q\bar{r}}\Phi^{qj}_{\ \ k}-g^{j\bar{r}}\frac{\partial}{\partial t}g_{q\bar{r}}\Phi^{iq}_{\ \ k}\right )\\
&+2\re \left(\frac{\partial}{\partial t}\Phi_{ij}^{\ \ k} \Phi^{ij}_{\ k}\right).
\end{align*}
Thus
\begin{align}
(\frac{\partial}{\partial t}-\Delta)S=&-|\overline{\nabla}\Phi|^2-|\nabla\Phi|^2+\Phi_{ij}^{\ \ k}\left(B_k^{\ q}\Phi^{ij}_{\ \ q}-B^{\ i}_q\Phi^{qj}_{\ \ k}-B_q^{\ j}\Phi^{iq}_{\ \ k}\right)\nonumber\\
&+2\re \left((\frac{\partial}{\partial t}-\Delta)\Phi_{ij}^{\ \ k}\Phi_{\ \ k}^{ij}\right).
\end{align}
where $B_i^j=g^{j\bar{r}}\frac{\partial}{\partial t}g_{i\bar{r}}+R^{\ p\ j}_{p\ i}$. From eqution (3.1) and formula (2.2), we get
\begin{align*}
\frac{\partial}{\partial t} g_{i\bar{j}}&=-R_{i\bar{j}}+\hat{R}_{i\bar{j}}+F_{i\bar{j}}(\varphi,z),\nonumber \\
R^{\ p\ q}_{p\ k}&=R_k^{\ q}-\nabla^p T_{pk}^{\ \ q}-\nabla_k T^{pq}_{\ \ p}
\end{align*}
Hence \begin{align}
B_i^{\ j}=g^{j\bar{r}}(\hat{R}_{i\bar{r}}+F_{i\bar{r}}(\varphi,z))-\nabla^p T_{pi}^{\ \ j}-\nabla_i T^{pj}_{\ \ p}.
\end{align}
Here\vspace{.5mm} $ F_{i\bar{r}}(\varphi,z)=F_{i\bar{r}}+F''\varphi_i\varphi_{\bar{r}}+F'\varphi_{i\bar{r}}+F'_i \varphi_{\bar{r}}+F'_{\bar{r}} \varphi_i.$\\
Now we compute the evolution of $\Phi_{ij}^{\ \ k}$. First
\begin{align*}\frac{\partial}{\partial t}\Phi_{ij}^{\ \ k}
&=g^{k\bar{l}}\nabla_i\frac{\partial}{\partial t}g_{j\bar{l}}\\
&=-\nabla_iR_j^{\ k}+g^{k\bar{l}}(\nabla_i\hat{R}_{j\bar{l}}+\nabla_iF_{j\bar{l}}(\varphi, z)).
\end{align*}
Note that \begin{align}\nabla_{\bar{q}} \Phi_{ij}^{\ \ k}=-R_{i\bar{q}j}^{\ \ \ k}+\hat{R}_{i\bar{q}j}^{\ \ \ k}.
\end{align}
Then
\begin{align*}
\Delta \Phi_{ij}^{\ \ k}&=-\nabla^{\bar{p}}R^{\ \ \ k}_{i\bar{p}j}+\nabla^{\bar{p}}\hat{R}^{\ \ \ k}_{i\bar{p}j}\nonumber\\
&=\nabla_i (-R^{\ k}_j+\nabla^qT_{qj}^{\ \ k}+\nabla_jT_{\ \ p}^{pk})-T_{iq}^{\ r}R^{\ q\ k}_{r\ j}+\nabla^{\bar{p}}\hat{R}_{i\bar{p}j}^{\ \ \ k} .\end{align*}
So we have
\begin{align*}(\frac{\partial}{\partial t}-\Delta)\Phi_{ij}^{\ \ k}&=\nabla_i(g^{k\bar{l}}(\hat{R}_{j\bar{l}}+F_{j\bar{l}}(\varphi, z))-\nabla^qT_{qj}^{\ \ k}-\nabla_jT_{\ \ p}^{pk})+T_{iq}^{\ r}R^{\ q\ k}_{r\ j}-\nabla^{\bar{p}}\hat{R}_{i\bar{p}j}^{\ \ \ k}\\
&=\nabla_iB_j^{\ k}+T_{iq}^{\ r}R^{\ q\ k}_{r\ j}-\nabla^{\bar{p}}\hat{R}_{i\bar{p}j}^{\ \ \ k}.
\end{align*}
Combining with $(3.16)$, we get
\begin{align*}
(\frac{\partial}{\partial t}-\Delta )S=&-|\overline{\nabla}\Phi|^2-|\nabla\Phi|^2+
\Phi_{ij}^{\ \ k}\left(B_k^{\ q}\Phi^{ij}_{\ \ q}-B^{\ i}_q\Phi^{qj}_{\ \ k}-B_q^{\ j}\Phi^{iq}_{\ \ k}\right)\\
&2\re \left(\nabla_iB_j^{\ k}+T_{iq}^{\ r}R^{\ q\ k}_{r\ j}-\nabla^{\bar{p}}\hat{R}_{i\bar{p}j}^{\ \ \ k}\right)\Phi^{ij}_{\ \ k}.
\end{align*}
As $T_{ij\bar{k}}=\hat{T}_{ij\bar{k}},$
\begin{align}
\nabla^pT_{pk}^{\ \ q}=g^{p\bar{l}}g^{q\bar{r}}(\hat{\nabla}_{\bar{l}}\hat{T}_{pk\bar{r}}-\Phi_{\bar{l}\bar{r}}^{\ \ \bar{s}}\hat{T}_{pk\bar{s}}).
\end{align}
By (3.17) $$|B_{i\bar{j}}|\leq k(t)(S^{1/2}+1+|\nabla \varphi|^2_g+|\varphi_{i\bar{j}}|_g^2)\leq k(t)(S^{1/2}+1),$$
Now we want to control $\nabla_i B_j^{\ k}$. From $(3.17)$ we need the following estimates from [19] obtained by similar calculations as (3.19),
$$|\nabla_i\nabla^qT_{qj}^{\ k}|\leq k(t)(S+|\overline{\nabla}\Phi|+1), $$
$$|\nabla_i\nabla_jT_{\bar{p}}^{\ k\bar{p}}|\leq k(t)(S+|\nabla\Phi|+1).$$
Also $$|T_{iq}^{\ r}R^{q\ k}_{r\ j}|\leq k(t)(|\overline{\nabla}\Phi|+1),$$
$$ |\nabla^{\bar{p}}R_{\bar{p}ij}^{\ \ \ k}|\leq k(t)(S^{1/2}+1).$$
We bound the terms with $\varphi_{ij}$ and $\Phi^{ij}_{\ \ k}$ in $\re (\nabla_iB_j^{\ k}\Phi^{ij}_{\ \ k})$ by $|\varphi_{ij}|^2 + k(t)S$. Together with the above estimates we get
\begingroup
\addtolength{\jot}{.4em}
\begin{align}
(\frac{\partial}{\partial t}-\Delta )S&\leq k(t)(S^{3/2}+S+1)+\sum_{i,j}|\varphi_{ij}|^2-\frac{1}{2}(|\nabla \Phi|^2+|\overline{\nabla} \Phi|^2).
\end{align}
\endgroup
We will use the similar way as in [16, 19] to control the term $S^{3/2}$.
The evolution equations below can be obtained by following the computation in [18, 19].
\begin{align}\begin{split}(\frac{\partial}{\partial t}-\Delta )\tr_{\hat{g}}g&\leq -\frac{S}{k_2(t)}+k_2(t),\\
(\frac{\partial}{\partial t}-\Delta )|\nabla \varphi|^2_{\hat{g}}&\leq -\sum_{i,j}\frac{|\varphi_{ij}|^2}{k_3(t)}+k_3(t), \end{split} \end{align}
Now we will apply a maximum priciple argument to the quantity \vspace{-1mm}$$H=\frac{S}{(C_1(t)-\tr_{\hat{g}}g)^2}+\frac{\tr_{\hat{g}}g}{C_2(t)}+\frac{|\nabla \varphi|_{\hat{g}}^2}{C_3(t)}.$$
Here we can take $C_i(t)$ to be the form of $Le^{\lambda Ce^{\alpha/t}}$ where $C,\ \alpha$ is the same as in (3.7) and $L, \lambda$ will be determined later. Let $L,\ \lambda>2$ such that \begin{align}\frac{C_1(t)}{2}\leq C_1(t)-\tr_{\hat{g}}g\leq C_1(t),\ \ \ 0<-\frac{C'_i(t)}{C^2_i(t)}\leq \frac{1}{\sqrt{C_i(t)}},\ \  i=1,2,3.
\end{align}
We calculate the evolution of $H$.
\begingroup
\addtolength{\jot}{.4em}
\begin{align*}
(\frac{\partial}{\partial t}-\Delta ) H=&\frac{1}{(C_1(t)-\tr_{\hat{g}}g)^2}(\frac{\partial}{\partial t}-\Delta )S+\frac{2S}{(C_1(t)-\tr_{\hat{g}}g)^3}(\frac{\partial}{\partial t}-\Delta )\tr_{\hat{g}}g\\
&-\frac{4\re\nabla\tr_{\hat{g}}g \cdot \overline{\nabla}S}{(C_1(t)-\tr_{\hat{g}}g)^3}-\frac{6S|\nabla \tr_{\hat{g}}g|^2}{(C_1(t)-\tr_{\hat{g}}g)^4}-\frac{2C'_1(t)S}{(C_1(t)-\tr_{\hat{g}}g)^3}\\
&+\frac{1}{C_2(t)}(\frac{\partial}{\partial t}-\Delta )\tr_{\hat{g}}g+\frac{1}{C_3(t)}(\frac{\partial}{\partial t}-\Delta )|\nabla \varphi|^2_{\hat{g}}-\frac{C_2'(t)\tr_{\hat{g}}g}{C_2(t)^2}-\frac{C_3'(t)|\nabla \varphi|^2_{\hat{g}}}{C_3(t)^2}.
\end{align*}
\endgroup
Taking $C_2(t), C_3(t)$ large enough and using (3.5), (3.7) and (3.22), the last two terms can be bounded by a constant C. Assuming $S> 1$ at the maximum point of $H$, from (3.20) we have \begin{align*}
(\frac{\partial}{\partial t}-\Delta )S&\leq k_1(t)(S^{3/2}+1)+\sum_{i,j}|\varphi_{ij}|^2-\frac{1}{2}|\overline{\nabla} \Phi|^2.
\end{align*}Together with (3.21), (3.22), we get
\begingroup
\addtolength{\jot}{.5em}
\begin{align*}
0\leq &\ (\frac{\partial}{\partial t}-\Delta ) H\\
\leq& \left(\frac{4k_1(t)}{C^2_1(t)}S^{3/2}+ \frac{4k_1(t)}{C^2_1(t)} +\frac{4}{C^2_1(t)}\sum_{i,j}|\varphi_{ij}|^2-\frac{|\overline{\nabla} \Phi|^2}{2C_1^2(t)}\right) +\left(-\frac{2S^2}{k_2(t)C_1^3(t)}+\frac{16k_2(t)S}{C_1^3(t)}\right)\\
&+\frac{4|\re\nabla\tr_{\hat{g}}g \cdot \overline{\nabla}S|}{(C_1(t)-\tr_{\hat{g}}g)^3}+\frac{2S}{\sqrt{C_1^3(t)}}
+\left(-\frac{1}{k_2(t)C_2(t)}S+\frac{k_2(t)}{C_2(t)}\right)\\
&+\left(-\frac{1}{k_3(t)C_3(t)}\sum_{i,j}|\varphi_{ij}|^2+\frac{k_3(t)}{C_3(t)}\right)+C
\end{align*}
\endgroup
As $|\nabla \tr_{\hat{g}}g|\leq \frac{1}{64}k_5(t)S^{1/2}$ an\vspace{1mm}d $|\overline{\nabla} S|\leq 2S^{1/2}|\overline{\nabla}\Phi|,$
\begin{align*}\frac{4|\re\nabla \tr_{\hat{g}}g \cdot \overline{\nabla} S|}{(C_1(t)-\tr_{\hat{g}}g)^3}&\leq \frac{k_5(t)S|\overline{\nabla}\Phi|}{C^3_1(t)}
 \leq\frac{|\overline{\nabla} \Phi|^2}{2C_1^2(t)}+ \frac{k^2_5(t)S^2}{2C_1^4(t)}.\end{align*}
We will also use \vspace{.6mm}
\begingroup
\addtolength{\jot}{.6em}
\begin{align*}
\frac{4k_1(t)S^{3/2}}{C^2_1(t)}&\leq \frac{S^2}{k_2(t)C_1^3(t)}+\frac{4k_1^2(t)k_2(t)S}{C_1(t)}.\end{align*}
\endgroup
Recall that all $k_i(t), C_i(t)$ are functions of the form $Le^{\lambda Ce^{\alpha/t}}$. First choose $C_i(t)>k_i(t)$, then fix $C_2(t), C_3(t)$. Now take the constant $L, \lambda$ in $C_1(t)$ to be large enough such that  $\frac{k^2_5(t)}{2C_1^4(t)}\leq \frac{1}{k_2(t)C_1^3(t)},\ \frac{4}{C^2_1(t)}\leq \frac{1}{k_3(t)C_3(t)}$ and $\frac{16k_2(t)}{C_1^3(t)}+\frac{2}{\sqrt{C_1^3}(t)}+\frac{4k_1^2(t)k_2(t)}{C_1(t)}\leq \frac{1}{2k_2(t)C_2(t)}$. The above estimates then give that at $(t_0, z_0)$,
$$0\leq \frac{-1}{2k_2(t)C_2(t)}S+C',$$
for some constant $C'$. Therefore $S\leq 4C'k_2(t)C_2(t)\leq C'C_1(t)$ at $(t_0, z_0)$.
It follows that $H$ is bounded by some constant $C$ depending only on $\sup|\varphi_0|$ and $\sup|\dot{\varphi}_0|$, which gives the desired estimate of $S$.
\end{proof}

Using $(3.7)$, the above estimate $S\leq C(t)$ implies that $\|\varphi(t)\|_{C^{2+\alpha}(M,g)}$ can be bounded by a smooth function $C(t)$ on $(0, T']$, which depends only on $\sup|\varphi_0|$ and $\sup|\dot{\varphi}_0|$. Differentiating the equation $(3.1)$ in $t$, we get
\begin{align}
\frac{\partial \dot{\varphi}}{\partial t}=\triangle \dot{\varphi}+ F'(\varphi,z)\dot{\varphi}
\end{align}To apply parabolic Schauder estimates to obtain higher order estimates, we still need to bound the derivatives of $g_{i\bar{j}}$ in the $t$-direction. Then it is sufficient to bound $|\Ric(g)|$.

\begin{lem}
There exists a smooth function $C(t)>0$ on $(0,T']$ depending only on $\sup |\varphi_0|$ and $\sup |\dot{\varphi}_0|$ and blowing up as $t\rightarrow 0$ such that $|\Ric|<C(t)$ for $t\leq T'$.
\end{lem}
\begin{proof}To compute the evolution of $|\Ric|$, first
$$\frac{\partial}{\partial t}R_{j\bar{k}}=-g^{l\bar{q}}\nabla_{\bar{k}}\nabla_j\frac{\partial}{\partial t}g_{l\bar{q}}\\
=-g^{l\bar{q}}\nabla_{\bar{k}}\nabla_j(-R_{l\bar{q}}+\hat{R}_{l\bar{q}}+F_{l\bar{q}}(\varphi,z))$$
Use $(2.1),(2.2)$ we have
\begingroup
\addtolength{\jot}{.4em}
\begin{align*}
\nabla_{\bar{k}}\nabla_jR_{l\bar{q}}&=\nabla_l\nabla_{\bar{q}} R_{j\bar{k}}-\nabla_l T_{\bar{k}\bar{q}}^{\ \ \bar{s}}R_{j\bar{s}}+T_{\bar{k}\bar{q}}^{\ \ \bar{s}}\nabla_lR_{j\bar{s}}+R^{\ \ \ r}_{l\bar{k}j}R_{r\bar{q}}\\
&\ \ \ -R^{\ \ \bar{s}}_{l\bar{k}\ \bar{q}}R_{j\bar{s}}+\nabla_{\bar{k}}T^{\ r}_{lj}R_{r\bar{q}}+T^{\ r}_{lj}\nabla_{\bar{k}}R_{r\bar{q}}
\end{align*}
\endgroup
So
\begingroup
\addtolength{\jot}{.4em}
\begin{align*}
(\frac{\partial}{\partial t}-\Delta ) R_{j\bar{k}}=&\nabla_{\bar{k}}T^{\ r}_{lj}R^{\ l}_r+T^{\ r}_{lj}\nabla_{\bar{k}}R^{\ l}_r+R^{\ \ \ r}_{l\bar{k}j}R^{\ l}_r\\
&-R^{\ \ \bar{s}l}_{l\bar{k}}R_{j\bar{s}}+\nabla^{\bar{q}}T_{\bar{k}\bar{q}}^{\ \ \bar{s}}R_{j\bar{s}}+T_{\bar{k}\bar{q}}^{\ \ \bar{s}}\nabla^{\bar{q}}R_{j\bar{s}}\\
&-g^{l\bar{q}}\nabla_{\bar{k}}\nabla_j(\hat{R}_{l\bar{q}}+F_{l\bar{q}}(\varphi,z)).
\end{align*}
\endgroup
From (3.19) we get $$|\nabla T|\leq k(t)(1+S^{1/2}), \ \ |\overline{\nabla} T|\leq k(t)(1+S^{1/2}).$$
where $S$ is bounded by some $k(t)$ by Proposition 3.3.
Note that $(3.18)$ gives
\begin{align}
R_{i\bar{j}}=\hat{R}_{i\bar{j}}+\nabla_{\bar{j}}\Phi_{ik}^{\ \ k},\ \ \  |\overline{\nabla}\Phi|\leq |\Rm|+ k(t).
\end{align}
Use this and similar calculation as (3.19) to get
$$|g^{l\bar{q}}\nabla_{\bar{k}}\nabla_j\hat{R}_{l\bar{q}}|\leq k(t)(|\Rm|+1),$$
Also we have $$|g^{l\bar{q}}\nabla_{\bar{k}}\nabla_jF_{l\bar{q}}(\varphi,z)|\leq k(t)(|\Rm|+1).$$
Therefore
\begingroup
\addtolength{\jot}{.4em}
\begin{align*}
|(\frac{\partial}{\partial t}-\Delta ) R_{j\bar{k}}|&\leq k(t)(|\nabla \Ric|+|\Rm|^2+|\Rm|+1)\\
&\leq k(t)(|\nabla \Ric|+|\Rm|^2+1).
\end{align*}
\endgroup
As $|\frac{\partial}{\partial t}g_{i\bar{j}}|=|-R_{i\bar{j}}+\hat{R}_{i\bar{j}}+F_{i\bar{j}}(\varphi,z)|\leq |\Ric|+ k(t)$, direct computation gives
\begin{align*}
(\frac{\partial}{\partial t}-\Delta ) |\Ric|^2&\leq k(t)(|\Ric|^3 +|\Ric|^2)+ 2|(\frac{\partial}{\partial t}-\Delta )\Ric||\Ric|-2|\nabla \Ric|^2.
\end{align*}
We then obtain the following
\begin{align*}
(\frac{\partial}{\partial t}-\Delta ) |\Ric|&=\frac{1}{2|\Ric|}(\frac{\partial}{\partial t}-\Delta ) |\Ric|^2+2|\nabla |\Ric||^2)\\
&\leq k_1(t)(|\nabla \Ric|+|\Rm|^2+1)-\frac{|\nabla \Ric|^2}{|\Ric|}+\frac{|\nabla |\Ric||^2}{|\Ric|}
\end{align*}
Let us consider $$H=\frac{|\Ric|}{C_1(t)}+\frac{S}{C_2(t)},$$ as in [18] where $C_1(t), C_2(t)$ are the functions of the form $Le^{\lambda Ce^{\alpha/t}}$ as in the proof of Propostion 3.3 such that $-\frac{C'_i(t)}{C^2_i(t)}\leq \frac{1}{\sqrt{C_i(t)}}, i=1,2 $. Assume $H$ achieves maximum at a point $(t_0, z_0)$, $t_0>0$, and assume $|\Ric|\geq 1$ at $(t_0, z_0)$. From (3.20) and Proposition 3.2, 3.3 we have
$$(\frac{\partial}{\partial t}-\Delta )S\leq -\frac{1}{2}Q+k_2(t)$$ where $Q=|\nabla \Phi|^2+|\overline{\nabla} \Phi|^2$. Take $C_1(t)>k_1(t), C_2(t)\geq \max\{S, S^2, k_2(t)\}$. Direct computation gives
\begingroup
\addtolength{\jot}{.4em}
\begin{align}
(\frac{\partial}{\partial t}-\Delta )H\leq\ &\frac{k_1(t)(|\nabla \Ric|+|\Rm|^2)}{C_1(t)}-\frac{|\nabla \Ric|^2}{C_1(t)|\Ric|}+\frac{|\nabla |\Ric||^2}{C_1(t)|\Ric|}\nonumber\\
&+\frac{|\Ric|}{\sqrt{C_1(t)}}+\left(-\frac{Q}{2C_2(t)}+\frac{k_2(t)}{C_2(t)}\right)+\frac{S}{\sqrt{C_2(t)}}\nonumber\\
\leq\ &\frac{k_3(t)|\Rm|^2}{C_1(t)}-\frac{|\nabla \Ric|^2}{2C_1(t)|\Ric|}+\frac{|\nabla |\Ric||^2}{C_1(t)|\Ric|}-\frac{Q}{2C_2(t)}+C
\end{align}
\endgroup
where the last inequality we use
\begingroup
\addtolength{\jot}{.4em}
\begin{align*}
\frac{k_1(t)|\nabla \Ric|}{C_1(t)}\leq \frac{|\nabla \Ric|^2}{2C_1(t)|\Ric|}+\frac{k_1^2(t)|\Ric|}{2C_1(t)}
\end{align*}
\endgroup
Using
$\nabla H=0$ at $(t_0, z_0)$ and $|\nabla |\Ric||\leq |\nabla\Ric|$, we get
 \begingroup
\addtolength{\jot}{.5em}
\begin{align*}
\frac{|\nabla |\Ric||^2}{C_1(t)|\Ric|}&=\frac{|\nabla S\cdot \overline{\nabla} |\Ric||}{C_2(t)|\Ric|}\\
&\leq \frac{|\nabla \Ric|^2}{2C_1(t)|\Ric|}+\frac{C_1(t)|\nabla S|^2}{2C_2^2(t)|\Ric|}\\
&\leq \frac{|\nabla \Ric|^2}{2C_1(t)|\Ric|}+\frac{C_1(t)k_4(t)Q}{C_2^2(t)|\Ric|}
\end{align*}
\endgroup
where\vspace{.8mm} the last inequality we use $ |\nabla S|^2\leq 2S(|\overline{\nabla}\Phi|^2+|\nabla\Phi|^2).$ From (3.23), $$|\Rm|^2\leq \frac{3}{2}Q+k_5(t),\ \ \ |\Ric|\leq \sqrt{Q}+k_6(t).$$
Choose $C_2(t)\geq 8k_4(t)$. Fix $C_2(t)$ and choose $C_1(t)\geq \max\{k_3(t)k_5(t), k_6(t)\}$ large enough such that $\frac{3k_3(t)}{2C_1(t)}\leq \frac{1}{4C_2(t)}$ and then fix $C_1(t)$.
Combining the above estimates, we obtain that at $(x_0,t_0)$,
$$
0\leq (\frac{\partial}{\partial t}-\Delta )H\leq -\frac{Q}{4C_2(t)}+\frac{C_1(t)Q}{8C_2(t)|\Ric|}+C'$$
for some constant $C'$. If $\frac{|\Ric|}{C_1(t)}\leq 1$, then $H\leq 2$ at $(t_0, z_0)$ and we obtain the estimate for $|\Ric|$. Otherwise at $(t_0, z_0)$
$$0\leq -\frac{Q}{8C_2(t)}+C'.$$
Therefore $Q\leq 8C'C_2(t)\leq 8C'C_1(t)$ at $(t_0, z_0)$. By our choice of $C_1(t), C_2(t)$, $H$ is bounded by some constant $C$ depending only on $\sup|\varphi_0|$ and $\sup|\dot{\varphi}_0|$ , which gives the bound for $|\Ric|$.
\end{proof}
The estimates we have obtained imply that the parabolic $C^{\alpha, \alpha/2}$ norm of the coefficients in equation (3.23) can be bounded. The parabolic Schauder estimates then give a $C^{2+\alpha, 1+\alpha/2}$ bound for $\dot{\varphi}$ in $[\epsilon, T']\times$ M for any $\epsilon>0$, with the bounds only depending on $\epsilon,\ \sup|\varphi_0|$ and $\sup|\dot{\varphi}_0|$. Similarly we can obtain a $C^{2+\alpha, 1+\alpha/2}$ bound for $\varphi_k, \varphi_{\bar{k}}$ in $[\epsilon, T']\times M$. Differentiate the flow again and repeat using Schauder estimate, we obtain all higher order estimates for $\varphi$. Let $\epsilon\rightarrow 0$, we obtain the bounds in Proposition 3.1 which blow up as $t\rightarrow 0$. Particularly, there exists a smooth solution on $[0,T]$ where $T$ is the same as in Lemma 3.1 and depends only on $\sup|\varphi_0|$.

\section{ proof of Theorem 1.1}
Assume that $\hat{\omega}$ satisfies the condition (1.3), then it follows from [13] that Ko{\l}odziej's stability result (Corollary 4.4 in [11]) is also true. In particular if
\vspace{2mm}
\begin{align*}
(\hat{\omega}+\sqrt{-1}\partial\bar{\partial}\phi_1)^n=(\hat{\omega}+\sqrt{-1}\partial\bar{\partial}\phi_2)^n=f\hat{\omega}^n,
\end{align*}
with $f\geq 0 \in L^p(M,\hat{\omega}),\ p>1$ and $\int_Mf\hat{\omega}^n=\int_M\hat{\omega}^n$,\vspace{1.3mm} then $\phi_1-\phi_2=const$. Now suppose that $\phi\in PSH(M,\hat{\omega})\cap L^\infty(M)$ is a weak solution of the equation \vspace{1.3mm}
\begin{align}
(\hat{\omega}+\sqrt{-1}\partial\bar{\partial}\phi)^n=e^{-F(\phi, z)}\hat{\omega}^n.
 \end{align}
 Then $f(z)=e^{-F(\phi(z), z)}\in L^p(M,\hat{\omega}),\ p>1$ as $\phi$ is bounded for $t\leq T$. Also the condition (1.3) gives that $\int_Mf\hat{\omega}^n=\int_M\hat{\omega}^n$. Therefore theorem 5.2 in [3] indicates that $\phi$ is continuous. Approximating $\phi$ with a sequence of smooth functions $\phi_j$  such that
\begin{align}
\sup_M|\phi_j-\phi|\rightarrow 0,\ \ {as}\ \ j \rightarrow \infty
\end{align}
It follows from [21] there exist smooth functions $\psi_j$ such that
\begin{align}
(\hat{\omega}+\sqrt{-1}\partial\bar{\partial}\psi_j)^n=c_je^{-F(\phi_j,z)}\hat{\omega}^n
 \end{align}
 where $c_j>0$ are constants chosen to satisfy the integration equality of the above equation. We normalize $\psi_j$ as in [11]
 \begin{align}
 \sup (\psi_j-\phi)=\sup(\phi-\psi_j),
\end{align}
 the stability result from [13] gives
\begin{align}
\lim_{j\rightarrow \infty} \|\psi_j-\phi\|_{L^{\infty}}=0.
\end{align}\\
Consider the equations \begin{align}\frac{\partial \varphi_j}{\partial t}=\log\frac{(\hat{\omega} + \sqrt{-1}\partial \bar{\partial}\varphi_j)^n}{\hat{\omega}^n}+F(\varphi_j, z)-\log c_j.\end{align}
Applying Proprosition 3.1, there exist a sequence of smooth functions $\varphi_j$ with $\varphi_j(0)=\psi_j$ such that $\varphi_j$ solves the equations on $[0,T_j]$ where $T_j$ only depends on $\sup |\psi_j|$ and $\sup |\dot{\varphi}_j(0)|$. Using (4.3) and (4.6)
\begin{align}
\dot{\varphi}_j(0)=F(\psi_j,z)-F(\phi_j,z).
\end{align}
It follows from (4.2) and (4.5) that $\sup |\psi_j|$ and $\sup |\dot{\varphi}_j(0)|$ can be bounded by a constant only depending on $\sup|\phi|$. Therefore there exists a $T>0$ independent of $j$ such that $\varphi_j$ solve the equation (4.6) on $[0,T]$.
By Lemma 3.1 in [18], $\{\varphi_j\}$ is a Cauchy sequence in $C^0([0,T]\times M)$. Let
$$\beta(t,z)=\lim_{j\rightarrow \infty} \varphi_j,$$
which is continuous on $[0,T]\times M$. For any $\epsilon >0$, from the proof of Proposition 3.1, we have bounds on all derivatives of $\varphi_j$ for $t\in [\epsilon, T]$. Then $\beta\in C^{\infty}([\epsilon, T]\times M)$ and
$$\lim_{j\rightarrow \infty} \|\beta-\varphi_j\|_{C^k([\epsilon, T]\times M)}=0.$$
Lemma 3.1 gives that $|\dot{\varphi_j}(t)|\leq \sup|\dot{\varphi_j}(0)|e^{Ct}$, for $t\in [0,T]$. From (4.7) we get
$$\dot{\varphi}_j(0)\rightarrow 0\ \ {as} \ \ j\rightarrow \infty.$$
 Therefore for any $t>0$,
$$\dot{\beta}(t)=\lim_{j\rightarrow \infty} \dot{\varphi}_j(t)=0.$$
As it is continuous on $[0,T]$, we have $\beta(0)=\beta(t)$ for $t\in (0,T]$ is smooth. But $\beta(0)=\lim_{j\rightarrow \infty} \varphi_j(0)=\lim_{j\rightarrow \infty} \psi_j=\phi$, thus we get the smoothness of $\phi$.

\begin{remk}
Note that Proposition 3.1 holds on any compact Hermitian manifolds. If Ko{\l}odziej's stability result can be extended to general Hermitian manifolds, then we can remove the assumption (1.3) in Theorem 1.1.
\end{remk}

\author{School of Mathematics, \\ \indent University of Minnesota, Minneapolis, MN, 55455}\\
\indent \emph{E-mail address}: \texttt{niexx025@umn.edu}


\begin{thebibliography}{1}
\bibitem{} E. Bedford and B.A. Taylor, The Dirichlet problem for a complex Monge-Amp\`{e}re equation. Invent. Math. 37 (1976), no. 1, 1{-}44.
\bibitem{} Z. B{\l}ocki, A gradient estimate in the Calabi-Yau theorem. Math. Ann. 344 (2009), no. 2, 317{-}327.
\bibitem{} S. Dinew and S. Ko{\l}odziej, Pluripotential estimates on compact Hermitian manifolds, arXiv:0910.3937.
\bibitem{} P. Gauduchon, Le th\`{e}or\`{e}me de l'excentricite nulle. (French) C. R. Acad. Sci. Paris S\`{e}r. A-B 285 (1977), no. 5, A387{-}A390
\bibitem{} M. Gill, Convergence of the parabolic complex Monge-Amp\`{e}re equation on compact
Hermitian manifolds, Comm. Anal. Geom. 19 (2011), no. 2, 277{-}303.
\bibitem{} B. Guan and Q. Li, Complex Monge-Amp\`{e}re equations on Hermitian manifolds, arXiv:0906.3548.

\bibitem{} A. Hanani, \'{E}quations du type de Monge-Amp\`{e}re sur les vari\'{e}t\'{e}s hermitiennes
compactes, J. Funct. Anal. 137 (1996), no. 1, {49\--75}.
\bibitem{} M. Klimek, Pluripotential Theory, Oxford University Press, New York, 1991.

\bibitem{} S. Ko{\l}odziej, The complex Monge-Amp\`{e}re equation. Acta Math. 180 (1998), no. 1, 69{-}117.

\bibitem{} S. Ko{\l}odziej, The complex Monge-Amp\`{e}re equation and pluripotential theory. Mem. Amer. Math. Soc. 178 (2005), no. 840
\bibitem{} S. Ko{\l}odziej, The Monge-Amp\`{e}re equation on compact K\"{a}hler manifolds. Indiana Univ. Math. J. 52 (2003), no. 3, 667{-}686.
\bibitem{} S. Ko{\l}odziej, H\"{o}lder continuity of solutions to the complex Monge-Amp\`{e}re equation with the right-hand side in $L^p$: the case of compact K\"{a}hler manifolds. Math. Ann. 342 (2008), no. 2, 379{-}386.
\bibitem{} S. Ko{\l}odziej, Email communication.

\bibitem{} D.H. Phong, N. \u{S}e\u{s}um and J. Sturm, Multiplier ideal sheaves and the K\"{a}hler-Ricci flow, Comm.
Anal. Geom. 15 (2007), no. 3, 613{-}632.
\bibitem{} D. H. Phong and J. Sturm, The Dirichlet problem for degenerate complex Monge-Amp\`{e}re equations. Comm. Anal. Geom. 18 (2010), no. 1, {145-170}.
\bibitem{} D.H. Phong, J. Song, J. Sturm and B. Weinkove, On the convergence of the modified K\"{a}hler-Ricci flow and solitons. Comment. Math. Helv. 86 (2011), no. 1, 91{-}112.
\bibitem{} J. Song and G. Tian, The K\"{a}hler-Ricci flow through singularities, arXiv:0909.4898.
\bibitem{} G. Sz\'{e}kelyhidi and V. Tosatti, Regularity of weak solutions of a complex Monge-Amp\`{e}re equation. Anal. PDE 4 (2011), no. 3, 369{-}378.
\bibitem{} M. Sherman and B. Weinkove, Local Calabi and curvature estimates for the Chern-Ricci
flow, arXiv:1301.1622.
\bibitem{} V. Tosatti and B. Weinkove, Estimates for the complex Monge-Amp\`{e}re equation on
Hermitian and balanced manifolds, Asian J. Math. 14 (2010), no. 1, {19-40}.
\bibitem{} V. Tosatti and B. Weinkove, The complex Monge-Amp\`{e}re equation on compact Hermitian manifolds, J. Amer. Math. Soc. 23 (2010), no. 4, 1187{-}1195.
\bibitem{} V. Tosatti and B. Weinkove, On the evolution of a Hermitian metric by its Chern-Ricci form, arXiv:1201.0312.
\bibitem{} V. Tosatti and B. Weinkove, The Chern-Ricci flow on complex surfaces,  arXiv:1209.2662.
\bibitem{} V. Tosatti and B. Weinkove, and X. Yang, Collapsing of the Chern-Ricci flow on elliptic surfaces,  arXiv:1302.6545.
\bibitem{} S.-T. Yau, On the Ricci curvature of a compact K\"{a}hler manifold and the complex Monge-Amp\`{e}re
equation I, Comm. Pure Appl. Math. 31 (1978), 339{-}411.\\
\end{thebibliography}
\end{document}